\documentclass{article}

\usepackage[
a4paper,
left=40mm,
top=25mm,bottom=25mm,right=25mm,]{geometry}

\usepackage{amsmath}
\usepackage{amssymb}
\usepackage{amsthm}

\usepackage{graphicx}
\usepackage{subfigure}
\usepackage{float}
\usepackage{booktabs}

\usepackage{hyperref}
\usepackage{cite} 

\hypersetup{colorlinks=true,
	linkcolor=blue,
	filecolor=magenta,      
	urlcolor=black,
	pdftitle={Overleaf Example},
	pdfpagemode=FullScreen,}
\newtheorem{prop}{Proposition}

\begin{document}
	\begin{center}
		{\large  \bf Bilingualism beyond a transitional state and its societal stability}\\
		\vspace{2mm}
		{\small Chami Kariyapperuma}\index{Kariyapperuma,Chami },\\ {\small Kushani De Silva}\index{De Silva, Kushani}\\
		{\small Research \& Development Center for Mathematical Modeling, Department of Mathematics, University of Colombo, Sri Lanka}\\
		{\small 2018s17093@stu.cmb.ac.lk}\index{Kariyapperuma,Chami }, {\small kdesilva@maths.cmb.ac.lk}\index{De Silva, Kushani}\\[2mm]
	\end{center}

\begin{abstract}
For the longest time, languages have been competing for their speakers to survive, although this problem has only recently gained rigorous attention from the scholarly community as a means to address the risk of losing speakers for the majority of the world's languages. With its exhaustive literature on the topic,  bilingualism has been a mode of discussion to overcome the death of minority language groups. However, in this context of bilingualism, language mixing is based on the assumption of equal competency in any human encounter, enabling successful communication. In this work, we point out that the former assumption may not be reasonable and explore a new avenue of language mixing strategy to study its impact on the extinction of a threatened language.\\

Keywords: language competition, language death, language-mixing, competency
\end{abstract}

\section{Introduction} 
There are more than 7000 languages being spoken around the world as of today. Due to accelerated globalization with the dominance of a few languages, many languages are at risk of going extinct \cite{williams2002language}. This occurrence poses a larger threat to global, cultural, and linguistic diversity \cite{fishman2001can,boissonneault2021systematic}. To address this serious problem, different mathematical models using differential equations to explain language dynamics have emerged; some of the popular works include \cite{isern2014language,minett2008modelling,mira2005interlinguistic,kandler2010language,strogatz2003modelling}. These models mainly looked at, in different contexts, how favorable languages recruit people resulting fewer people in less favorable (minority) languages. Many of these models were birthed from Abrams and Strogatz language model (AS) that stressed the risk of language death of minority languages although its limitations were addressed in Isern-Fort model (IF) \cite{strogatz2003modelling,isern2014language}. These language models made an effort to showcase threatened languages such as Celtic, Gaelic, Welsh, and Quechua and discussed the possibilities of rebounding from extinction. However, recent work that used native American extinct languages such as Texas German showcased that IF model was not able to capture the actual dynamics of a language that has gone extinct, motivating the scholars to explore more avenues to understand the dynamics behind extinction \cite{thisanka2022can}.

With this constant attention to threatened (minority) languages, different models have paid attention to having a bilingual population generated by mixing two monolingual languages. In literature on two language societies, the bilinguals have been defined explicitly as a separate language group (e.g. \cite{minett2008modelling,mira2005interlinguistic,heinsalu2014role,vazquez2010agent}), as well as implicitly in two language population models where language group classification was done as speakers and non-speakers of a language (e.g. \cite{isern2014language,iriberri2012minority,patriarca2012modeling}). Recent work has showcased how bilingualism is on the rise in India defeating a well-popular high-status language, English \cite{de2020higher}. These evidences show that bilingualism can be on the rise in many parts of the regions of the world due to exponential growth in globalization. However, all these models where bilingualism is taken into account have assumed that bilinguals are equally competent and failed to explicitly discuss the capacity of communication. Their assumption is that any two persons who are knowledgeable in more than one language in a well-mixed society of two languages are able to communicate without any struggle. This may not be realistic because when a person is transitioning from one monolingual to another monolingual only that person will become bilingual. This is a continuous process when looked at from one person's side of the transition. Thus it is not realistic to think that any two persons who meet can communicate using both languages in full capacity as needed. Thus in this study, we focus on the competency level of a language mapped into the language competition model. Henceforth we introduce a novel communication approach built into the language model from which the dynamics of language shift be based on the ability to communicate rather than the knowledge of vocabulary. This ability will be based on a degree of the mutuality of the usage of a particular language between the two individuals. 

\section{Modeling background of Bilingualism}
Language competition models proposed in \cite{minett2008modelling,mira2005interlinguistic,heinsalu2014role,vazquez2010agent} have defined the bilingual group explicitly as a separate group in a two-language society. In this case, the bilingual group has been incorporated into the model predominantly in two ways. One is when the bilingual group acts merely as a bridge between the two monolingual groups, under the assumption that the attracting population of the bilingual group (individuals that influence speakers to move from one monolingual group to bilingual group) only consists of individuals from the other monolingual group, rather than bilinguals themselves (e.g. \cite{minett2008modelling}). The second instance is when the bilinguals are considered as a representative/agent of both languages, and is therefore included in the attracting population of both the monolingual groups and/or the bilingual group based on the assumptions of each model (e.g. \cite{mira2005interlinguistic,heinsalu2014role,vazquez2010agent}). Both of these instances suggest that the transition between language groups has been modeled under the impression that a monolingual is attracted towards the bilingual group mainly because of their attraction towards the other language (the other monolingual language group), rather than bilingualism itself. Here we discuss them in detail to build the need for the new communication concept. 

\subsection{Wang and Minnet, 2008}
This model examines bilingualism explicitly by introducing a third language group that speaks both of the languages X and Y. Then the society consists of three language groups X, Y, and Z (bilingual group); with respective population fractions given by $x,y$, and $z$, with $x+y+z=1$. Out of the six transitions between the language groups, two transitions $X \longrightarrow Y$ and $Y \longrightarrow X$ were excluded due to their rarity. 

This model embodies transitions of two sorts; vertical and horizontal. Vertical transmission implies the transference of language from parents to their children. Based upon a simplified version of this concept, the uni-parental vertical model (V-model) was adopted (see Table \ref{MW_table1}). A horizontal transmission model (H-model) was adopted that implies the transference between adults of language groups (see Table \ref{MW_table2}).

\begin{table}[hbt]
	\centering
	\caption{V-model of Wang and Minnet}
	\begin{tabular}{p{22mm}p{30mm}p{21mm}p{36mm}}
		\toprule
		State of \newline parent & parent$\rightarrow{}$offspring & State of \newline offspring & Transition probability \newline $P_{ij}$ for $i \to j$  \\
		\midrule
		Monolingual & $X\rightarrow{X}$ & Monolingual & 1\\
		{} & $Y\rightarrow{Y}$ & {} & 1 \\
		
		Bilingual & $Z\rightarrow{X}$ & Monolingual & $c_{ZX}s_Xx^a$\\
		{} & $Z\rightarrow{Y}$ & {} & $c_{ZY}s_Yy^a$\\
		
		Bilingual & $Z\rightarrow{Z}$ & Bilingual & $1- c_{ZX}s_Xx^a -c_{ZY}s_Yy^a$ \\
		\bottomrule
	\end{tabular}
	\label{MW_table1}
\end{table}
Here the transition probability for children of bilingual parents to become monolingual in X is taken as,
\begin{equation}
	\label{MW_eq1}
	P_{ZX}=c_{ZX}s_Xx^a,
\end{equation}
where $s_X$ is the status of language X (with $s_x+s_Y=1$), and $c_{ZX}$ (peak rate at which speakers of Y switch to speak X) , $a$ are constants, with $x^a$ being the attractiveness of X . Similar behavior can be seen for offspring of bilinguals becoming monolinguals in Y.

\begin{table}[hbt]
	\centering
	\caption{H-model of Wang and Minnet}
	\begin{tabular}{p{22mm}p{30mm}p{21mm}p{36mm}}
		\toprule
		State of adult (before) & parent$\rightarrow{}$offspring & State of adult (after) & Transition probability\newline $P_{ij}$ for $i \to j$  \\
		\midrule
		Bilingual & $Z\rightarrow{Z}$ & Bilingual & 1\\
		Monolingual & $X\rightarrow{Z}$& Bilingual & $c_{XZ}s_Yy^a$ \\
		{} & $Y\rightarrow{Z}$& {} & $c_{YZ}s_Xx^a$\\
		Monolingual & $X\rightarrow{X}$ & Monolingual  & $1-c_{XZ}s_Yy^a$\\
		{} & $Y\rightarrow{Y}$ & {} & $1-c_{YZ}s_Xx^a$\\
		\bottomrule
	\end{tabular}
	\label{MW_table2}
\end{table}
Here the transition probability for a monolingual in X to become Bilingual is taken as,
\begin{equation}
	\label{MW_eq1}
	P_{XZ}=c_{XZ}s_Yy^a,
\end{equation}
where $s_Y$ is the status of language Y (with $s_x+s_Y=1$), and $c_{XZ}$ (peak rate at which speakers of X switch to speak Z) , $a$ are constants, with $y^a$ being the attractiveness of Y. Here it is important to note that the factor that contributes to a monolingual in X to become bilingual is the attractiveness of Y (attracting population that only consist of speakers of Y) and the status of Y, rather than those of Z, implying that in this model bilingual group doesn't have an inherent status or doesn't act as attraction population of itself, but only a bridge between the two monolingual language groups. 

Ultimately a unified model was built by combining the V-model and H-model while defining a mortality rate, $\mu$ at which adults are replaced by children, to the model.

\begin{align}
		\frac{dx}{dt}&=\mu zP_{ZX}-(1-\mu)xP_{XZ} \notag \\
		\frac{dy}{dt}&=\mu zP_{ZY}-(1-\mu)yP_{YZ} \label{MW_eq2}
\end{align}

Moreover, by substituting $z=1-x-y$ and incorporating the attraction of different language groups corresponding to respective transmissions (Eq. \eqref{MW_eq1}), this model was further simplified, and it ultimately predicted that one of the two competing languages will eventually acquire all the speakers, regardless of the initial conditions, resulting in a monolingual system in which only one language is spoken.

Henceforth, Minnet and Wang went on to introduce an agent-based model for the above-mentioned phenomena, where an agent corresponds to the speakers of the population who can adopt either monolingual or bilingual states. In order to build this new model, the same formulae in the previous model (Eqs. \eqref{MW_eq1},\eqref{MW_eq2}) were re-interpreted in such a way that they specify the probabilities with which each agent makes the transition from state to state, instead of the rates of change of the proportions of individuals having certain states. Minnet and Wang used this model to claim that by intervention and appropriate changes to parameters, the language shift can be altered such that both languages persist.

\subsection{Mira and Paredes, 2005}
This model considers the three language groups X, Y, and B where X and Y are monolinguals and B is bilingual, with respective population fractions as x,y, and b, with $x+y+b=1$. This model is a generalized/extended version of AS model with the parameter of similarity between competing languages(k), where $k=0$ represents the instance where no communication is possible between monolingual speakers, which reflects the cases chosen in AS model. In this model, the rate of change of $x$ was given by,
\begin{equation} 
	\frac{dx}{dt}=yP_{YX} + bP_{BX} -x(P_{XY} +P_{XB})
\end{equation}
with analogous equations for $dy/dt$ and $db/dt$; where $P_{ij}$ is the transition probability between groups $i$ and $j$ per unit time, $i,j=X,Y,B$. The transition probability is given by,
\begin{align}
	P_{YB} &= cks_x(1-y)^a \notag \\
	P_{YX} &= c(1-k)s_X(1-y)^a  \label{mp_eq}  
\end{align}
where $c,a$ are constants. 
While the bilingual status is not included in the model due to the idea that bilinguals do not possess inherent characteristics, it is important to recognize that bilingual speakers serve as representatives of both languages and thus play a role in attracting others to become bilingual. As a result, bilinguals are part of the attracting population that impacts the transition from monolingual language group (X or Y) to bilingual group (B), as well as the transitions from X to Y and Y to X among monolingual groups.

In this model bilinguals also capture the dissimilarity of the two monolingual groups. In other words, the bilingual group is not a state of transition but a state that can be achieved by a monolingual group, i.e.
\begin{equation}
	\lim_{k \to 1} Y=X; \text{ by } Y\to B \text{ and } X \to B
\end{equation}
However, this work has failed to state a proper quantification of the concept similarity between monolingual groups. It can possibly be the similarity of vocabulary to the best of our understanding. Thus the bilingual state interpreted in this model reflects the mutuality of the vocabulary of two monolingual languages.

\subsection{Heinsalu et al., 2014}

This model modifies Wang and Minnet \cite{minett2008modelling} such that the rate of change of a monolingual of X(or Y) becoming a bilingual is proportional to the total number of speakers of language Y(or X) including bilinguals, i.e., to the sum $N_Y +N_Z$ (or $N_X +N_Z)$; where $N_X, N_Y$ and $N_Z$ denote the population fraction of respective language groups. They claimed that even if a language goes extinct within a monolingual group, it persists or even increases within the bilingual group.

\subsection{Vázquez et al., 2010}

This model has extended the Abrams-Strogatz model \cite{strogatz2003modelling} by including a third group of individuals who are bilingual and labeled as state Z. Monolingual users X and Y can become bilingual with a probability that depends on the number of their neighbors who are monolinguals of the opposite language, based on the idea that monolinguals are forced to become bilingual if they want to communicate with monolingual users of the opposite language. The model prohibits direct transitions from one class of monolingual to the other. Similarly, the transition from a bilingual Z to a monolingual X or Y depends on the number of neighbors using language X or Y, including bilingual agents.

\begin{align}
		& P(X \rightarrow Z)=(1-S) \sigma_y^a, \notag \\
		& P(Z \rightarrow Y)=(1-S)\left(1-\sigma_x\right)^a, \notag \\
		& P(Y \rightarrow Z)=S \sigma_x^a, \notag \\
		&P(Z \rightarrow X)=S\left(1-\sigma_y\right)^a, \label{vaz_eq}
\end{align}

where $\sigma_x$, $\sigma_y$, and $\sigma_z$ represent the densities of speakers in neighboring states X, Y, and Z, while $S$ represents the prestige of language X.

They had claimed that shift from language coexistence to dominance of one language occurs as the volatility parameter ($a$) reaches a critical value; with high volatility leading to coexistence, while low volatility leading to dominance or extinction of one language.

\section{New Model Set-up of Language Competition}
\label{ref:2.1_set up}
In this study, we have considered a society that has identified with two languages. We assumed the linguistic dissimilarity of the two languages $M_1$ and $M_2$ to be high so that two monolinguals, each belonging to $M_1$ and $M_2$ respectively, cannot have a conversation with each other.  Moreover, it is of importance to notice that, in actual social contexts, the transition of an individual between the two monolingual language groups is not immediate.  Instead, they would become bilingual once they pick up the second language, at least to a certain extent. When and if they eventually lose fluency in the first language only, they would become monolingual in the second language (see Fig. \ref{transition_diagram}). In that stage, the transition between $M_1$ and $M_2$ will be complete. Thus this well-mixed society with two languages will generate three distinct language groups -  $M_1$, $M_2$, and $B$ such that $m_1+m_2+b=1$ where lowercase letters represent the population fractions in each of the language groups, respectively. The definitions of the three language groups are given below:
\begin{table}[ht]
	\begin{tabular}{cl}
		Monolinguals & \begin{tabular}[c]{@{}p{10cm}@{}}Monolingual groups say $M_i$ ($i=1,2$) is defined such that an individual in group $M_i$ can communicate only using the language $M_i$ and nothing else.\end{tabular} \\
		Bilinguals   & \begin{tabular}[c]{@{}p{10cm}@{}}The bilingual group consists of people who can communicate using both $M_1$ and $M_2$.\end{tabular}                                                       
	\end{tabular}
\end{table}

\begin{figure}[th]
	\centering
	\includegraphics[width=100mm]{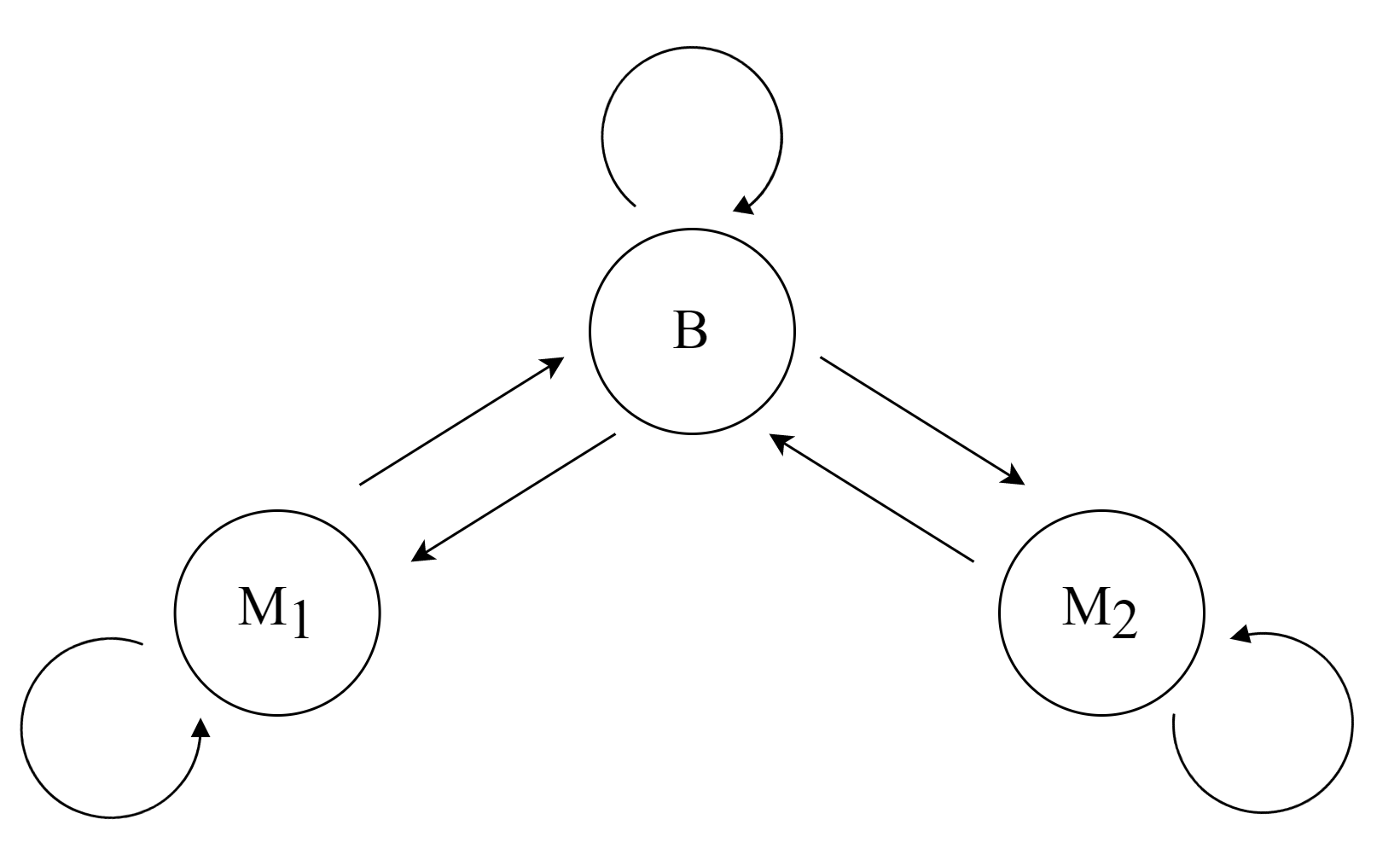}
	\caption{The language groups and possible transitions within a well-mixed bilingual society. Here Monolingual and bilingual groups are denoted respectively with $M_1$, $M_2$ and $B$.}
	\label{transition_diagram}
\end{figure}

\subsection{Novel classification of transition}
\label{sec:nct}
In accordance with the paradigm of the language groups given above, the transitions between them come down to the communication that would take place between any two individuals from two different groups. This ability to communicate is based upon the competency (proficiency) levels of a language, and these levels are usually assessed using four domains: speaking, listening, writing, and reading, all of which rely heavily on one's vocabulary knowledge. \cite{richards2002methodology}. Henceforth, we considered it the measurement of language competency (proficiency) in our study, with our assumption that individuals belonging to Monolingual groups (either $M_1$ or $M_2$) have 100\% vocabulary knowledge of that particular language. Therefore communication that takes place within $M_i$ for $i=1,2$ is not of concern with respect to the communication. However, the communication between $M_1$ and $M_2$ is of interest - because that gives birth to different levels of bilingualism. 

We mathematically denote an individual's competency level in language $j$ with $C_j$ where $0\leq C_j\leq 1$. A bilingual person who is 100\% competent in both languages is said to have achieved perfect bilingualism, i.e $C_{M_1}=C_{M_2}=1$. Other times it is $0< C_{M_1},C_{M_2}< 1$ when a bilingual person's competence is less than 100\% in at least one of the languages and is referred to as an imperfect bilingual. 

A bilingual person is typically characterized based on their level of language proficiency in each of the two languages. However, communication takes place neither based on one's vocabulary nor their similarity as in Wang and Minnet \cite{minett2008modelling} but based on the ability to communicate. In other words, simply knowing the vocabulary does not help someone in communicating because successful communication consists of  two individuals and a mutual set of vocabulary. This ability will be based on a degree of the mutuality of the vocabulary knowledge of a particular language by the two speakers. Thus we have proposed a novel classification approach to be applied in the process of classifying the population into three language groups; $M_1, M_2 \text{ and }B$. 


Bilingual's communication can be classified into the following instances;
\begin{enumerate}
	\item \label{NCC-ins1}When a bilingual is talking with a monolingual of $M_1$, the conversation is conducted only in $M_1$, depending on bilingual's competency level in $M_1$ ($C_{M_1}$) 
	\item \label{NCC-ins2}When a bilingual is talking with a monolingual of $M_2$, the conversation is conducted only in $M_2$, depending on the bilingual's competency level in $M_2$ ($C_{M_2}$) 
	\item \label{NCC-ins3}When a bilingual is talking with another bilingual, the conversation is conducted in accordance with the degree of mutuality of the competencies 
	(we denote the degree of mutuality using $x_{M_1},x_{M_2}$, see Eqs. \eqref{com1},\eqref{com2} for reference)
\end{enumerate}
In the above first and second instances, the only way the two individuals can have a conversation is to conduct it utilizing the bilingual's competency level of the monolingual language. But in the third instance where two bilinguals are involved, there are many ways they can utilize their language competencies to conduct a conversation.   
If the two bilinguals are perfect bilinguals, they use both languages ($M_1$ and $M_2$) interchangeably to communicate because the degree of mutuality is 100\%. If the conversation happens between a perfect bilingual and an imperfect bilingual, or two imperfect bilinguals, they can use both languages based on different mixing strategies, depending on different circumstances. Thus the communication takes place based on the degree of mutuality in vocabulary rather than the vocabulary of an individual. Based on these mixing strategies, the speakers will move through the language groups.

\subsection{The social status of a bilingual ($s_B$)}
\label{bilingual_status}
The social status of a language is assessed at junctures of prestige versus loyalty, urban versus rural social settings, and open versus closed societies while taking geographical characteristics, socioeconomic payoffs, and social and psycho-social mindsets into account \cite{heinsalu2014role}. On this basis, the parameter $s_i;$ $s_i\in[0,1]$ was defined as the status of the language $i$, measured in terms of socio-economic as well as cultural gains. In this model, we assume the status does not change over a long period of time, i.e.,
status is fixed for the time frame we look at. 

Let us denote the competency levels of two bilinguals, namely Person 1 ($P_1$) and Person 2 ($P_2$), in both languages in $C_{i,j}$ where $i$ is the person ($i=P_1, P_2$) and $j$ is the language where $j=M_1, M_2$ (see Table \ref{table:Comp level}). 
\begin{table}[hbt]
	\caption{Competency levels of two random bilinguals}
	\label{table:Comp level}
	\centering
	\begin{tabular}{lll}
		\toprule
		{} & \multicolumn{2}{c}{Language competency} \\
		\midrule
		Bilingual person & $M_1$ & $M_2$ \\
		\midrule
		$P_1$ & $C_{P_1M_1}$ & $C_{P_1M_2}$ \\
		$P_2$ & $C_{P_2M_1}$ & $C_{P_2M_2}$ \\
		\bottomrule
	\end{tabular}
	\label{table1}
\end{table}

Here, $C_{ij}$ denotes the competency of person $i$ in language $j$, where $i=P_1, P_2$ and $j=M_1, M_2$. The extent to which they could use the vocabulary knowledge of each language,  $M_1$ and $M_2$, to hold the conversation would be given by the following mutuality of vocabulary.\\ 
\begin{align}
	\label{com1}
	x_{M_1}=\text{min}(C_{P_1M_1},C_{P_2M_1})\\
	\label{com2}
	x_{M_2}=\text{min}(C_{P_1M_2},C_{P_2M_2})
\end{align}
The aforementioned scenario can be elaborated using an example as given in Table \ref{table2}. In this scenario where two persons ($P_1$ and $P_2$) have a conversation, they have the ability to use 60\% of $M_1$ and 50\% of $M_2$. \\
\begin{table}[hbt]
	\caption{Numerical example for the competency levels of two randomly selected bilinguals}
	\centering
	\begin{tabular}{lcc}
		\toprule
		{} & \multicolumn{2}{c}{Language competency}\\
		\midrule
		Bilingual person & $M_1$ & $M_2$ \\
		\midrule
		$P_1$  & 80\% & 50\% \\
		$P_2$  & 60\% & 70\% \\
		\bottomrule
	\end{tabular}
	\label{table2}
\end{table}
Thus we define the social status of a bilingual as a proportion of social status of monolingual languages. This proportion is determined by the degree of vocabulary one would use from each of the two monolingual languages (see Eq. ~\eqref{bi_equation}). 
\begin{equation}
	\label{bi_equation}
	s_B = s_{M_1}x_{M_1}+s_{M_2}x_{M_2}, \quad s_B \in (0,1]
\end{equation}
where, $x_{M_i}$ is the degree of mutuality in $M_i$ language and $s_{M_1}, s_{M_2} \in (0,1)$ respectively represent the social status of each language $M_1,M_2$. According to Eq.~\eqref{bi_equation}, a bilingual
person can become a perfect bilingual, i.e. $s_B = 1$ by having a competency of 100\% in both languages. By becoming a perfect bilingual, he/she can obtain a higher social status than either two monolingual groups. This movement eventually drives the individuals to become bilinguals and essentially avoid a lower status (either $M_1$ or $M_2$) going extinct. This has been proven through the model presented in this work (see Eq.~\eqref{ref:model_eq_2}).

\subsection{ODE Model System} 
\label{sec:gv1}
In this section, we propose a three-dimensional language model explaining the transitions between the language groups $M_1$, $M_2$ and $B$ as depicted in Fig. \ref{transition_diagram}. The ode model system that explains the dynamics of the three language groups is given below:\\
\begin{align}
		\dot m_1&=\,bP_{B,M_1}-m_1P_{M_1,B}, \notag \\
		\dot m_2&=\,bP_{B,M_2}-m_2P_{M_2,B}, \notag \\
		\dot b&=\,m_1P_{M_1,B}+m_2P_{M_2,B}-bP_{B,M_1}-bP_{B,M_2}, \label{ref:model_eq_1}
\end{align}
where $m_1$, $m_2$ and $b$ represent the population proportions of language groups $M_1$, $M_2$, and $B$ respectively, as described in detail in Section 3. The derivatives $\dot m_1$, $\dot m_2$, $\dot b$ are with respect to time and represent the change of population proportions $m_1$, $m_2$ and $b$ with time.
$P_{i,j}$ with $i, j = M_1,M_2, \text{ and }B$ represents the transition probabilities from group $i$ to $j$. We define the transition probabilities as follows:
\begin{equation}\label{trans_prob_func}
	P_{i,j} =\lambda s_j j^\alpha i^\beta,
\end{equation}
with $s_j \in (0,1]$ is the status of group $j$ and $\lambda>0$ is a scaling factor. The parameter $\alpha; \alpha \geq 1$ is the ease of attraction and that is contributing to the ``attracting population'' by means of $j^{\alpha}$ whereas $\beta;\beta\geq 1$ is the ease of survival which is contributing to the ``withdrawing population" one belongs to by means of $i^{\beta}$. Thus the full model in Eq.~\eqref{ref:model_eq_1} becomes,\\
\begin{align}		\label{ref:model_eq_2}
	\begin{cases}
		\vspace*{2mm}
		\dot m_1 &=\,\lambda s_{M_1} m_1^\alpha b^{\beta+1}-\lambda s_{M_2} b^\alpha m_1^{\beta+1}, \quad m_1(0)>0,\\ \vspace*{2mm}
		\dot m_2 &=\,\lambda s_{M_2} m_2^\alpha b^{\beta+1}-\lambda s_{M_1} b^\alpha m_2^{\beta+1}, \quad m_2(0)>0,\\ \vspace*{2mm}
		\dot b&=\,\lambda s_B b^\alpha m_1^{\beta+1}+\lambda s_B b^\alpha m_2^{\beta+1}-\lambda s_{M_2} m_1^\alpha b^{\beta+1}-\lambda s_{M_2} m_2^\alpha b^{\beta+1}, \quad b(0)>0.
	\end{cases}
\end{align}

This system suggests that the tendency of an individual to move through groups $i$ to $j$ happens along three axes: (a) social status acquired from group $j$ and benefits entitled thereupon (b) ease of attraction and (c) ease of survival in the existing group $i$. 

\section{Dynamical Analysis}
The system in Eq.~\eqref{ref:model_eq_2} has seven equilibria, three of which are trivial and boundary ($E_1$, $E_2$, $E_3$), and four of which are non-trivial ($E_4$, $E_5$, $E_6$, $E_7$) out of which $E_7$ is an interior equilibrium. The points are given below where $E_i$ shows the points in the order of $m_1^*,m_2^*,b^{*}$.
\begin{equation*}
	\begin{aligned}
		\label{eqpt_list}
		E_1 : &(1, 0 , 0),
		\\E_2 : &(0, 1 , 0),
		\\E_3 : &(0, 0 , 1),
		\\
		E_4 : &(m_1^*,1-m_1^*,0),
		\\
		E_5 : &\left(\frac{\left(\frac{s_{M_1}}{s_{B}}\right)^\delta}{1 + \left(\frac{s_{M_1}}{s_{B}}\right)^\delta}, 0, \frac{1}{1 + \left(\frac{s_{M_1}}{s_{B}}\right)^\delta}\right),
		\\
		E_6 : &\left(0 ,\frac{\left(\frac{s_{M_2}}{s_{B}}\right)^\delta}{1 + \left(\frac{s_{M_2}}{s_{B}}\right)^\delta}, \frac{1}{1 + \left(\frac{s_{M_2}}{s_{B}}\right)^\delta}\right),
		\\E_7 :  &\left(\frac{(\frac{s_{M_1}}{s_B})^\delta}{1 + (\frac{s_{M_1}}{s_B})^\delta + (\frac{s_{M_2}}{s_B})^\delta}, \frac{(\frac{s_{M_2}}{s_B})^\delta}{1 + (\frac{s_{M_1}}{s_B})^\delta + (\frac{s_{M_2}}{s_B})^\delta}, \frac{1}{1 + (\frac{s_{M_1}}{s_B})^\delta + (\frac{s_{M_2}}{s_B})^\delta}\right),
	\end{aligned}
\end{equation*}
where $\delta=1/(-\alpha+\beta+1)$. Stability analysis for the above equilibria was conducted using Jacobi stability analysis (see Appendix \ref{jacobi_stability}) in conjunction with phase portrait analysis, to get the full picture of the behavior of the system at each equilibrium. From the stability analysis, it was noticeable that $\alpha$ and $\beta$ play a major role in deciding the stability of the equilibria. For simplicity, we define a threshold $d$ to explain the stability for $\alpha-\beta$, a reflective functional trade-off between survival of one language and attraction of another. \\

In this analysis, we fix the status of each monolingual  at a reasonable value to indicate the relative social status; high-status or low-status. This is depicted in Table \ref{table:fix para}. We then changed the status of bilinguals in the numerical simulations as it plays a significant role in reflecting the degree of mutuality in either $M_1$ or $M_2$ (see Eq.~\eqref{bi_equation}).

\begin{table}[hbt]
	\caption{Parameter values in system \eqref{ref:model_eq_2} chosen for numerical simulations.}
	\centering
	\begin{tabular}{cc}
		\toprule
		Model parameter & Numerical value \\ \hline 
		$s_{M_1}$     &     0.3\\
		$s_{M_2}$     &     0.7\\
		$\lambda$   & 400 \\ \bottomrule
	\end{tabular}
	\label{table:fix para}
\end{table}

\subsection{Coexistence of all language groups ($E_7$)}
At the equilibrium where all languages coexist, stability is guaranteed if the following conditions are met according to trace of the Jacobian (see Appendix \ref{chap:appendixA2})., i.e. $Tr(J)<0$.\\
\begin{equation}
	\label{ineq1}
	\left(\frac{1}{-\delta}\right)\left(\frac{s _{M_1}}{s_B}\right)^{-\delta}\left(\frac{s_{M_2}}{s_B}\right)^{-\delta}
	\left(\frac{s_{M_{1}}^{\frac{\beta+1}{-\alpha+\beta+1}} }{ s_B^{\frac{\alpha}{-\alpha+\beta+1}} }+\frac{s_{M_{2}}^{\frac{\beta+1}{-\alpha+\beta+1}}}{s_B^{\frac{\alpha}{-\alpha+\beta+1}}}+\frac{s_{M_{1}}^{\frac{\beta}{-\alpha+\beta+1}}}{s_{B}^{\frac{\alpha-1}{-\alpha+\beta+1}}}+\frac{s_{M_{2}}^{\frac{\beta}{-\alpha+\beta+1}}}{s_{B}^{\frac{\alpha-1}{-\alpha+\beta+1}}}\right)<0
\end{equation}
\begin{align}
	\label{1neq1_1}
	&\left(\frac{s _{M_1}}{s_B}\right)^{-\delta} > 0 \\
	\label{1neq1_2}
	&\left(\frac{s_{M_2}}{s_B}\right)^{-\delta} > 0 \\  
	\label{1neq1_3}
	&\left(\frac{s_{M_{1}}^{\frac{\beta+1}{-\alpha+\beta+1}}}{s_B^{\frac{\alpha}{-\alpha+\beta+1}}}+\frac{s_{M_{2}}^{\frac{\beta+1}{-\alpha+\beta+1}}}{s_B^{\frac{\alpha}{-\alpha+\beta+1}}}+\frac{s_{M_{1}}^{\frac{\beta}{-\alpha+\beta+1}}}{s_{B}^{\frac{\alpha-1}{-\alpha+\beta+1}}}+\frac{s_{M_{2}}^{\frac{\beta}{-\alpha+\beta+1}}}{s_{B}^{\frac{\alpha-1}{-\alpha+\beta+1}}}\right) > 0
\end{align}

We numerically found that $E_7$ is globally stable for all ICs that are positive and for any value of $s_B$ given that the threshold $d$ is in the range of $(d\approx0.75\pm 0.15)$ i.e. $0.5\lessapprox d\lessapprox 0.9$. The uncertainty $0.15$ is produced by the range of $s_B$. This is depicted in Fig. \ref{fig: E7_1}. In that, any positive initial point in the space of $m_1-m_2$ converges to the $E_7$. Fig. \ref{fig: E7_1} additionally showcased the positions of other equilibria for full clarification. We changed the value of $s_B$ in the full spectrum of $\left( 0,1\right]  $ to observe the behavior of the equilibrium. The dynamic of the equilibrium with the change of $s_B$ is given in Fig. \ref{E7_diff} along with the population dynamics of the bilingual population. 
\begin{figure}[!h]
	\hfill
	\subfigure[]
	{\includegraphics[width=6.5cm]{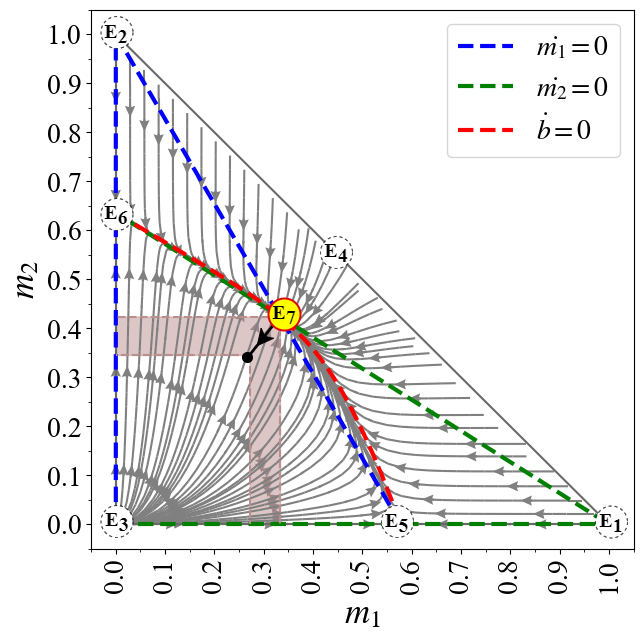}}
	\hfill
	\subfigure[]{\includegraphics[width=5.98cm]{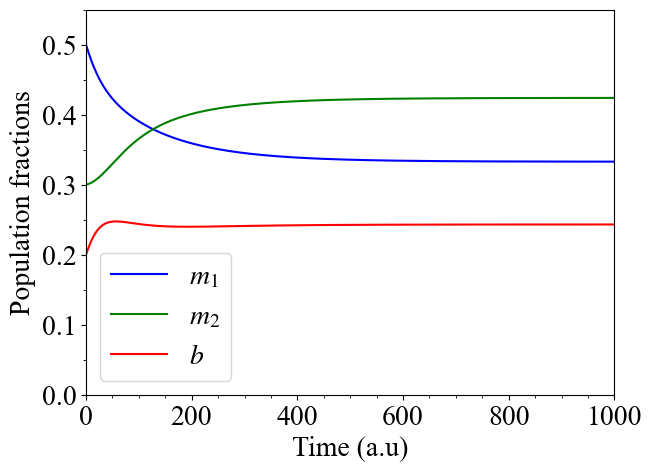}}
	\hfill
	\caption{Population dynamics of the system when $s_{M_1}=0.3, s_{M_2}=0.7, s_{B}=0.1$ is depicted. Here $\alpha-\beta=-2.5<1$ (where $\alpha=1.1$ and $\beta=3.6$) resulted in stable $E_7$. The left and right panels respectively show the phase portrait and the evolution of population fractions of the three groups (For this instance, ICs: $m_1=0.5, m_2=0.3$ and $b=0.2$).}
	\label{fig: E7_1}
\end{figure}
\begin{figure}[h!]
	\hfill
	\subfigure[]
	{\includegraphics[width=6cm]{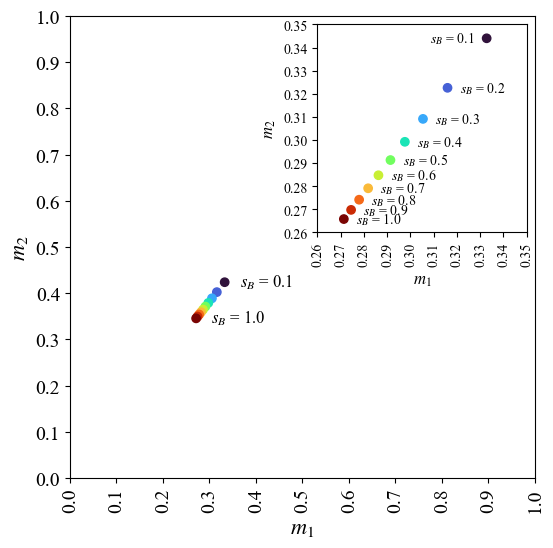}}
	\hfill
	\subfigure[]{\includegraphics[width=6.4cm]{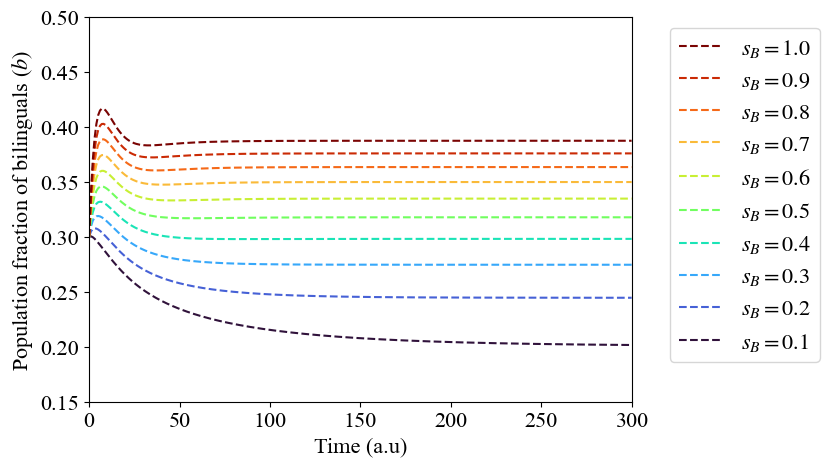}}
	\hfill
	\caption{(a) The movement of the equilibrium $E_7$ with $s_B$ is depicted. The point labeled $s_B=0.1$ corresponds to the phase portrait in Fig.\ref{fig: E7_1}. As $s_B$ goes from 0 to 1 both $m_1$ and $m_2$ decreases implying the increase of the population fraction of the bilinguals ($b$). (b) The time evolution of the bilingual population fraction corresponding to the movement of the equilibrium $E_7$ shown in Fig.\ref{E7_diff} is depicted. The trajectory labeled $s_B=0.1$ corresponds to the evolution of the population shown in the right panel of Fig.\ref{fig: E7_1}. As $s_B$ goes from 0 to 1, the population fraction value of the bilinguals ($b$) at the stable equilibrium $E_7$ increases.}
	\label{E7_diff}
\end{figure}

Overall, the system showed similar qualitative behavior for all bilingual statuses for the given parameter values. Moreover, this result is generalized as long as one language has lower status and the other monolingual has higher status ($s_{M_1}>s_{M_2}$ or vice versa). However, the range of the threshold can be changed for other possible values of the status of monolingual languages.

\subsection{Death of one language group}
In this section, we discuss the possibility of one language group going extinct: (a) lower status, (b) higher status, and (c) bilingual. These cases are established under equilibria $E_6, E_5$, and $E_4$ respectively. The stability conditions for these equilibria can be derived from Eq. \eqref{ineq1} in which the status value will be zero corresponding to zero population group (see Appendix \ref{stability;cond}).\\

The situation of lower status language disappearing occurs when the $\alpha -\beta$ is within the band of $(d,1)$ where $d$ is dependent of status of bilinguals. Recall that this $d$ was numerically found to be in ($0.5\lessapprox d\lessapprox 0.9$) such that $E_7$ obtains stability when $\alpha -\beta<d$. At this narrow band, the bifurcation takes place to the dynamics in Fig. \ref{fig: E7_1} and generate Fig. \ref{E2_sB_0.6} resulting convergence to $E_6$, for certain $s_B$ values. In particular, the stability of $E_6$ is attained when $s_{M_2}=s_B$ at $\left( \alpha-\beta\right) \in A=\left\lbrace d: \left( d,1\right) \right\rbrace $ as well as when $s_{M_1}<s_B<s_{M_2}$ or $s_{M_2}<s_B$, at $\left( \alpha-\beta\right) \in A=\left\lbrace d: \left( d,1\right) \smallsetminus d\in 1^-\right\rbrace$. On another note, as the status of $s_B$ increases, the $E_6$ moves down the $m_2$ axis. i.e. as more people are recruited to the bilingual group, the $m_2$ group loses its population fraction. However when $s_{M_2}<s_B$, $E_6$ converges to the equilibrium $E_3$ where both monolingual language groups disappear. \\
\begin{figure}[H]
	\hfill
	\subfigure[]{	\includegraphics[width=0.5\textwidth]{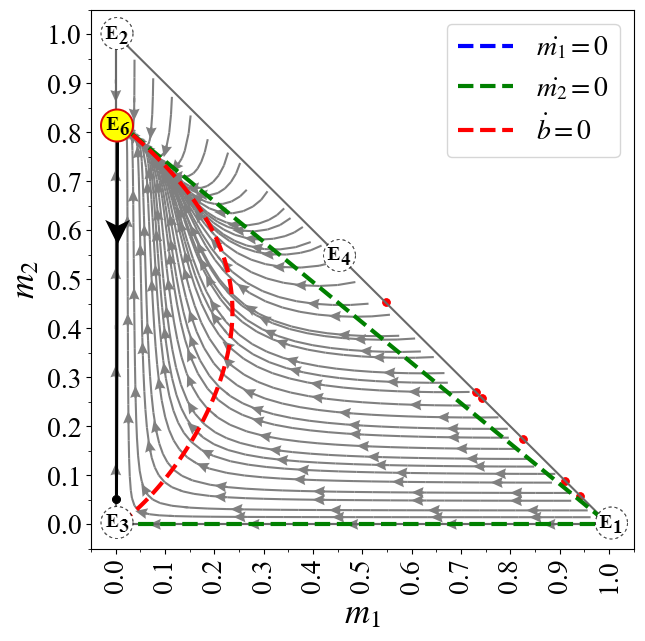}}
	\hfill
	\subfigure[]{\includegraphics[width=0.47\textwidth]{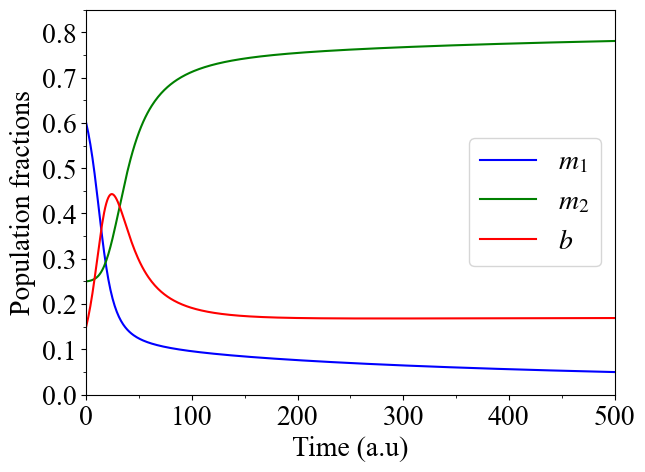}}
	\hfill
	\caption{The scenario where the lower status language disappears is depicted. Population dynamics of the system when $s_{M_1}=0.3, s_{M_2}=0.7, s_{B}=0.6$ and $\alpha-\beta=0.9$ (where $\alpha=2$ and $\beta=1.1$). The left panel shows the phase portrait and the right panel shows the population dynamics with ICs: $m_1=0.6, m_2=0.25$, and $b=0.15$.}
	\label{E2_sB_0.6}
\end{figure}

The stability conditions for the equilibrium point $E_5$ where only the higher status language going extinct (see Appendix \ref{stability;cond}) was found to be analogous to that of $E_6$ (see Appendix \ref{stability;cond}), with only $s_{M_2}$ being replaced by $s_{M_1}$. It could be observed that the lower status language group goes extinct at instances when its status is the lowest out of all three. But, since $s_{M_2}$ will never be the lowest out of the three statuses by definition, we can conclude that $E_5$ does not stably exist in this system.  \\

The disappearance of the bilingual group was observed for $s_B<s_{M_2}$ at certain $\alpha-\beta$ values. In particular, for bilingual status less than or equal to $s_{M_1}$, i.e., $s_{M_1} \geq s_B < s_{M_2}$, when $\left( \alpha-\beta\right)\in (d,1)$, dynamics of the system change such that the bilingual group collapses. This is depicted in Fig. \ref{E4_sB_0.1} for $\left( \alpha-\beta\right)=0.9 \in (d,1)$. The movement of the trajectories of the phase portrait shows that the system reaches different stable equilibria for different initial conditions, but satisfying $m_1^*+m_2^*=1$. The other instance where the bilingual group collapses occurs when $s_{M_1}< s_B < s_{M_2}$ but only for the limiting case of $\left(\alpha-\beta\right)$. In particular, when $s_{M_1}< s_B < s_{M_2}$ and $\alpha-\beta=0.9999$ (i.e. $\left( \alpha -\beta\right)  \to 1^{-1}$), the dynamics converge to $E_4$ where bilingual group disappears (Fig. \ref{E4_sB_0.5} portraits a stable $E_4$). The x,y space of phase portrait Fig. \ref{E4_sB_0.5} is divided into two domains by the nullcline $\dot b=0$ ($E_3E_4$) where the area above the nullcline represents $m_2>m_1$ and the area below the nullcline representing $m_2<m_1$. Here also the system reaches different stable equilibria for different initial conditions, but satisfying $m_1^*+m_2^*=1$ and $m_2>m_1$.

In summary, when the bilingual group goes extinct, the transition dynamics collapse because the only way the monolingual groups are connected is through bilingual group (see Fig. \ref{transition_diagram}).\\
\begin{figure}[H]   
	\centering
	\hfill
	\subfigure[]{\includegraphics[width=0.5\textwidth]{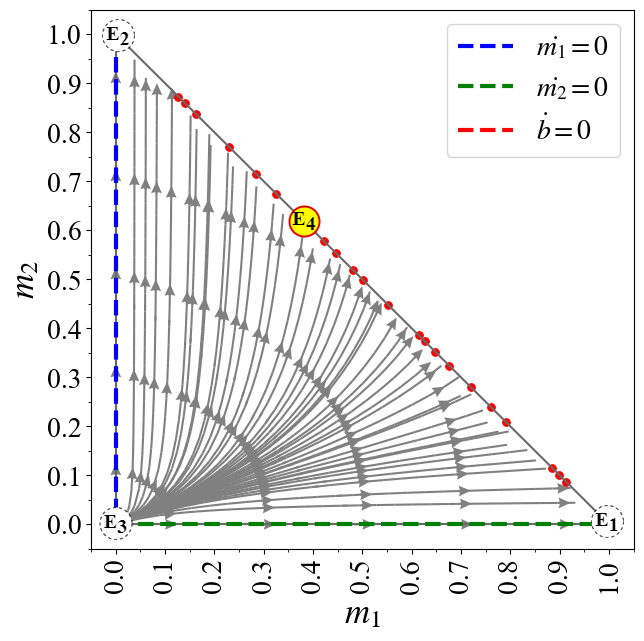}}         
	\hfill 
	\subfigure[]{\includegraphics[width=0.48\textwidth]{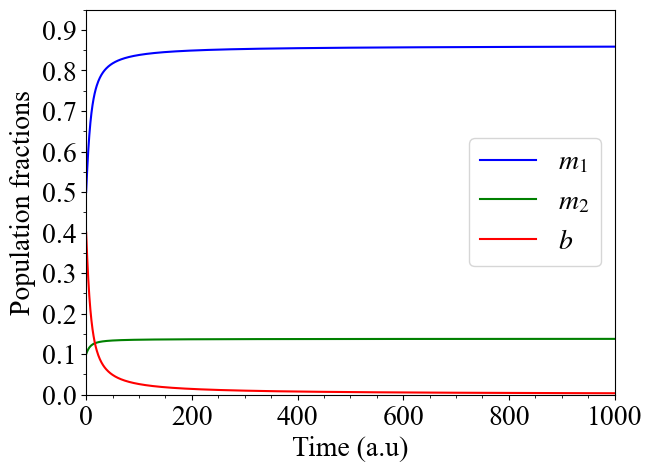}}
	\caption{(a): Phase portrait of the system when $s_{M_1}=0.3\geq s_{B}=0.1 < s_{M_2}=0.7$, and $\alpha-\beta=0.9$ (where $\alpha=2$ and $\beta=1.1$) where the points converge to $E_4$, i.e. bilingual disappear. The multiple steady states of $E_4$ are located along the diagonal of the x,y space, and only one steady state is labeled as $E_4$ to maintain the clarity of the diagram. (b): Dynamics of population fractions for the conditions given in (a) and ICs $m_1=0.5$, $m_2=0.1$, $b=0.4$.}
	\label{E4_sB_0.1}
\end{figure}

\begin{figure}[H]     
	\centering    
	\hfill
	\subfigure[]{\includegraphics[width=0.5\textwidth]{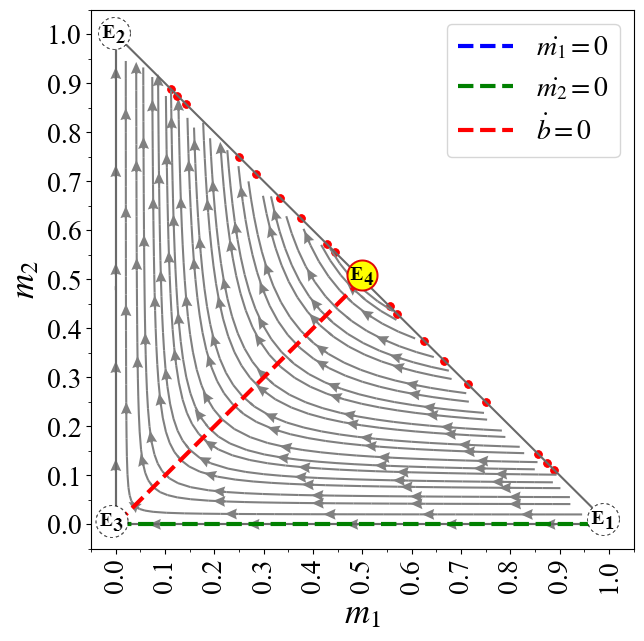}}   
	\hfill
	\subfigure[]{\includegraphics[width=0.48\textwidth]{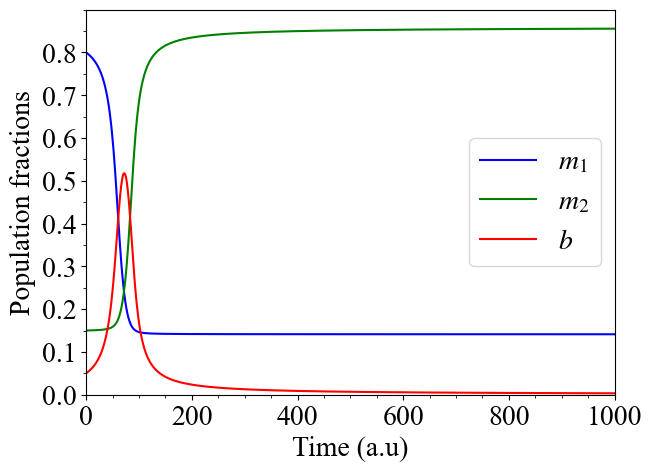}}
	\caption{(a): Phase portrait of the system when $s_{M_1}=0.3 < s_{B}=0.5 < s_{M_2}=0.7, $ and $\left( \alpha-\beta\right)  \to 1^{-}$ (where $\alpha=2.0999$ and $\beta=1.1$). This situation generate $E_4$ where bilingual group disappears. (b): The dynamics of population fractions under the same conditions of (a) and with ICs: $m_1=0.8$, $m_2=0.15$, $b=0.05$.}
	\label{E4_sB_0.5}
\end{figure}

With regard to first two equilibria $E_1,E_2$, only initial condition that would potentially end up reaching the equilibria $E_1$ and $E_2$ are the ones where $m_2=0$ and $m_1=0$ respectively. But the system is not defined at these points, therefore no population dynamics can be seen, and the stable existence of these equilibria does not occur in this system.\\
\subsection{Status dependent future of bilinguals}
Another situation of this society is where only the bilingual group exists. This is given by the equilibrium $E_3$ (see Fig. \ref{E7_sB_0.99}). This takes place only when ${s_B}>{s_{M_2}}$ and $\left( \alpha - \beta\right)  \to 1^{-}$ (in contrast to the case of bilingual going extinct when ${s_B}<{s_{M_2}}$ and $\left( \alpha - \beta\right)  \to 1^{-}$). Thus depending on the status of bilingual group against the status of the higher status language (given $\left( \alpha - \beta\right)  \to 1^{-}$) the future of the bilingual group as the only surviving group is determined. 
\begin{figure}[H]
	\hfill
	\subfigure[]{\includegraphics[width=0.46\textwidth]{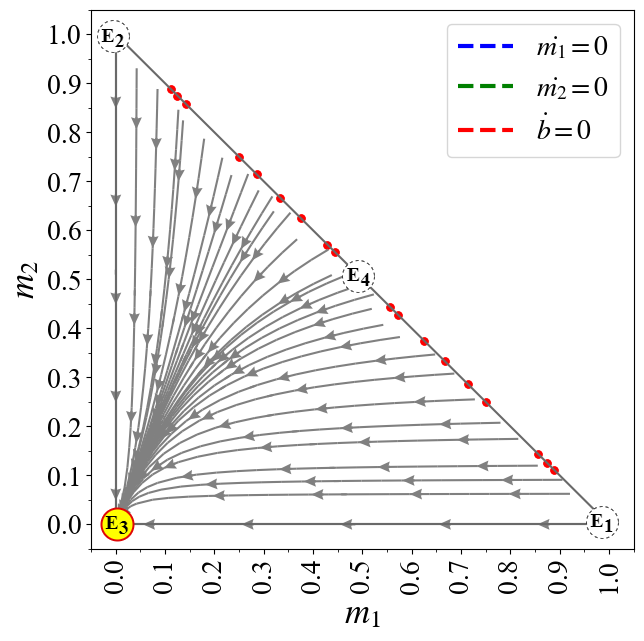}}
	\hfill
	\subfigure[]{\includegraphics[width=0.46\textwidth]{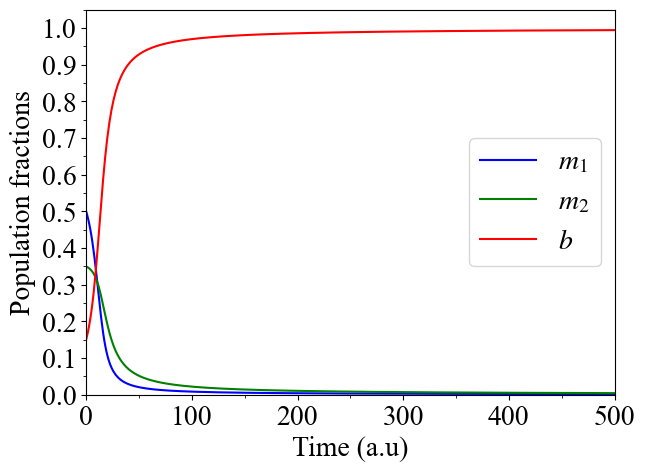}}
	\hfill
	\caption{Population dynamics when $s_{M_1}=0.3, s_{M_2}=0.7, s_{B}=0.9$ and $\alpha-\beta=0.9999$ (where $\alpha=2.0999$ and $\beta=1.1$). (b): instance when the two monolingual groups have higher initial population fractions than the bilingual group, but still eventually go extinct by losing their speaker to the bilingual group.}
	\label{E7_sB_0.99}
\end{figure}

\subsection{Local stability of $E_3$ and $E_4$}
\label{E7&E4}
The above scenarios were discussed for $\left( \alpha-\beta \right) <1$. However, when $\alpha-\beta>1$, the system shows the local stability simultaneously for $E_3$ and $E_4$, for any bilingual status ($s_B$) value. The respective phase portraits are depicted in Fig \ref{E7E4_sB_0.2_phase}.
\begin{figure}[H]
	\hfill
	\subfigure[]{\includegraphics[width=0.46\textwidth]{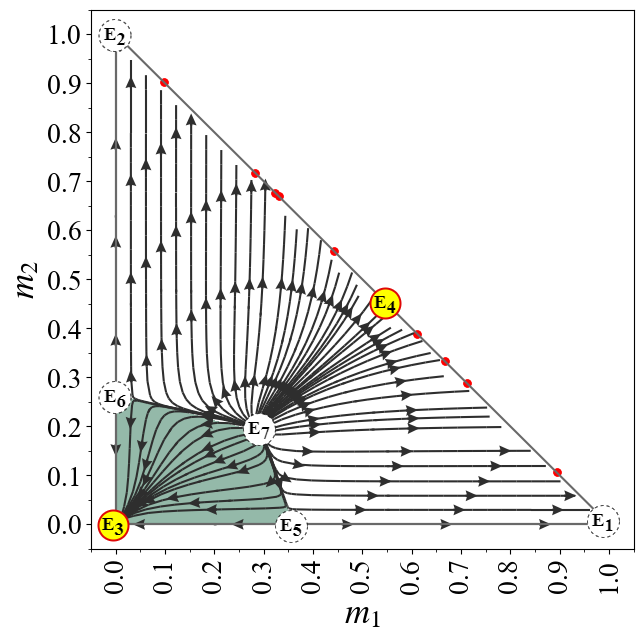}}
	\hfill
	\subfigure[]{\includegraphics[width=0.46\textwidth]{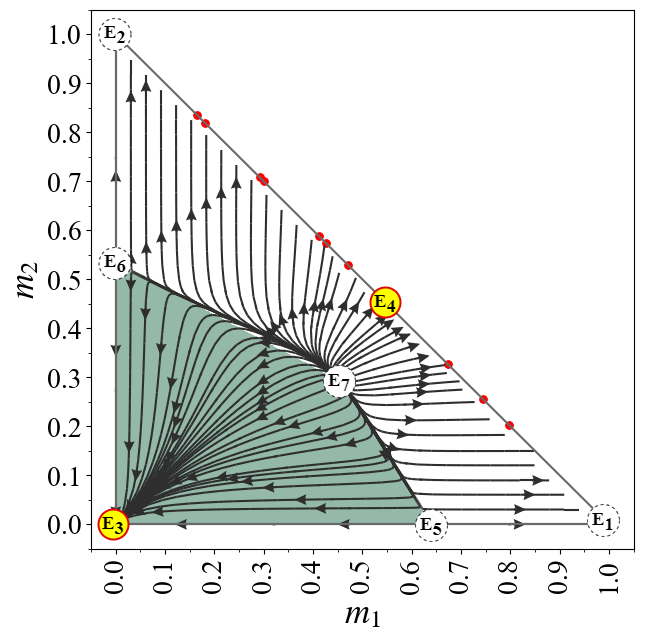}}
	\hfill
	\hfill
	\caption{(a) Population dynamics when $s_{M_1}=0.3, s_{M_2}=0.7, s_{B}=0.1$ and $\alpha-\beta=2.9$  (where $\alpha=4$ and $\beta=1.1$), (b) Population dynamics when $s_{M_1}=0.3, s_{M_2}=0.7, s_{B}=0.9$ and $\alpha-\beta=2.9$  (where $\alpha=4$ and $\beta=1.1$).}
	\label{E7E4_sB_0.2_phase}
\end{figure}

It was observed that the x,y space of phase portraits in Fig. \ref{E7E4_sB_0.2_phase} are separated into two domains by nullclines based on the direction of trajectories: domain (1)-the area below $E_6E_7E_5$ whose initial points reach $E_3$ steady state and domain (2)-the area over $E_6E_7E_5$ whose initial points reach $E_4$ steady states. It was also noticeable that with the increase in the bilingual status, the area of the attractor domain of $E_3$ increased.

\section{Conclusion}

Survival of a language becomes difficult when an attracting factor of another language becomes dominant. Mathematical models to study extinction risk and survival of minority languages have been on discussion recently. At the same time bilingualism is on the trend with the globalization and authors have looked at how bilinguals play a role in the extinction risk. \\

In this work we proposed a mathematical model \eqref{ref:model_eq_2} which represents the dynamics of language competition of a society with two monolingual groups and one bilingual group. In that we assumed the transition between monolingual groups are impossible unless through a bilingual group. Thus interaction between the two monolingual groups creates the bilinguals (see Fig. \ref{transition_diagram}). However in our model we assumed that the status of bilingual is a separate entity that is determined by its recruitments. In particular these recruitments are based on a degree of mutuality, i.e. two monolinguals can converse only based on their mutual set of vocabulary (Section \ref{bilingual_status} and therein). With the increment of the degree of mutuality, the status of the bilingual rises. With that definition of bilingual status, we define the transitional functions based on power functions of two parameters; the society's tendency to get attracted to new languages ($\alpha$) and the society's ability to survive using the language they already know ($\beta$).  With the equilibrium analysis of our model, we identified that these two parameters play a major role in the dynamical analysis of the model. Thus in the numerical study, we used the parameters, $\alpha$, $\beta$, and the bilingual status ($s_B$) as model variables. \\

Through the numerical study, it was found that the parameters $\alpha$ and $\beta$ impact the system as a pair, i.e., $\left( \alpha-\beta\right) $, which can be interpreted as functional trade-off between survival of one language and attraction of another. For the ease of grasping to the reader, the summary of all dynamics, we give here in conclusion of the dynamical analysis along the one dimensional axis of $\left(\alpha-\beta \right) \in \mathbb{R}$. Global stability of all three language groups is achieved when $\left(\alpha-\beta \right) <d$ such that $d$ was numerically found to be approximately in the range of $(d\approx0.75\pm 0.15)$ i.e. $0.5\lessapprox d\lessapprox 0.9$. The uncertainty range of $0.15$ is resulted by the spectrum of values of status of bilinguals, $s_B$ (see Figs. \ref{fig: E7_1},\ref{E7_diff}). In contrast to the above, when $\left( \alpha-\beta\right)>1$ , simultaneous local stability of $E_3$ (persistence of only the Bilingual group) and $E_4$ (persistence of only the two monolingual groups) is observed (see Fig. \ref{E7E4_sB_0.2_phase}). The parameter range $d<\left( \alpha-\beta\right) <1$ is when the system is undergoing bifurcation emerging new dynamics. Within this narrow band of $\alpha-\beta$ values, the system transitions causing the disappearance of certain equilibria and the system reaching stability at other equilibria in the process. Disappearance and emergence of the equilibria within this band of $\alpha-\beta$ values depend on the status of the bilinguals ($s_B$);
\begin{description}
	\item[$M_1\to 0$] Lower status language group disappears when
	\begin{itemize}
		\item $s_{M_2}=s_B$
		\item $s_{M_1}<s_B<s_{M_2}$ except when $\left( \alpha-\beta\right) \to 1^-$ (see Fig. \ref{E2_sB_0.6})
		\item $s_{M_2}<s_B$ except when $\left( \alpha-\beta\right) \to 1^-$
	\end{itemize}
	\item[$M_1,M_2 \to 0$] Both monolingual groups disappear
	\begin{itemize}
		\item $s_{M_2}<s_B$ only when $\left( \alpha-\beta\right) \to 1^-$ (see Fig. \ref{E7_sB_0.99})
	\end{itemize}
	\item[$B\to 0$] Bilingual group collapses
	\begin{itemize}
		\item $s_{M_1}\geq s_B<s_{M_2}$ (see Fig. \ref{E4_sB_0.1})
		\item $s_{M_1}<s_B<s_{M_2}$ only when $\left( \alpha-\beta\right) \to 1^-$ (see Fig. \ref{E4_sB_0.5})
	\end{itemize}
\end{description}

The above behavior patterns might imply that the societies whose aggregate behavior $\alpha-\beta$ lies between $d$ and $1$ can be considered to be at a threshold that is susceptible to a dramatic change in the face of a minor change in $\alpha-\beta$ (i.e., a change in the behavior patterns of the society), and reach either of the two behaviors, that is; the stable coexistence of all the language groups, or the simultaneous local stability of $E_3$ and $E_4$, that was found to show consistency throughout a wider range of $\alpha-\beta$ values. These results shed light on the significance of encouraging bilingualism in communities where a language with low status is in danger of extinction due to the dominance of a high-status language. Bilingualism can be promoted in such societies by emphasizing the fact that being a bilingual carries its own set of benefits, which could even be more advantageous than those of the high-status language group.\\


As a closing remark, this study shows that by incorporating a bilingual group into a society with two languages, it is possible to preserve the low-status language, either as a separate group of monolingual speakers of the low-status language, or within the bilingual group, for all the valid parameter values considered. This emphasizes the importance of promoting bilingualism in societies where a low-status language is threatened by a high-status language, and therefore faces the risk of going extinct. As future avenues, this concept of modeling can be tested on extinct languages to see if extinction is better explained for the already extinct languages. 


\appendix
\section{Appendix 1} 
\label{chap:appendixA1}

\subsection{Jacobi stability analysis}
\label{jacobi_stability}

The Jacobi stability analysis was conducted by computing the Jacobian matrix for the equilibria of the system (see Eq.~\eqref{ref:model_eq_1}), along with the corresponding eigenvalues, which were used to determine the stability of the fixed points. An equilibrium was considered asymptotically stable if all eigenvalues other than 0 have negative real parts, and unstable if at least one eigenvalue has a positive real part (see Proposition~\ref{prop1}).

\begin{prop}
	\label{prop1}
	Instances where one of the eigenvalues becomes 0, were noticeable in the stability analysis of the system depicted in Eq.~\eqref{ref:model_eq_1}. The eigenvalue that is equal to 0 is related to the condition (constraint) that 
	\begin{equation}
		\label{constraint}
		m_1+m_2+b=1
	\end{equation}
	Therefore the stability of the system is governed by eigenvalues other than 0.
\end{prop}

\begin{proof}
	This idea can be explained by comparing two approaches of defining the model and obtaining the eigenvalues. 
	
	\textit{Approach 1:} Presenting the model equations (see Eq.~\eqref{ref:model_eq_1}) as a three-dimensional system, with Eq.~\eqref{constraint} included as an additional constraint:
	\begin{align}
			\dot m_1 &= f(m_1,m_2,b) \notag \\
			\dot m_2 &= g(m_1,m_2,b) \notag \\
			\dot b &= h(m_1,m_2,b) \notag \\
			m_1&+m_2+b=1  \label{eig0_modeleq}
	\end{align}
	The Jacobian matrix for the above system (see Eq.~\eqref{eig0_modeleq}) is:
	
	\begin{equation}
		\label{jac_f,g,h_1}
		J_1=\left[\begin{matrix}f_{m_1} & f_{m_2} & f_{b}\\g_{m_1} & g_{m_2} & g_{b}\\h_{m_1} & h_{m_2} & h_{b}\end{matrix}\right]
	\end{equation}
	The eigenvalues of the above Jacobian matrix satisfies the following characteristic polynomial:
	
	\begin{equation}
		\label{char_eq}
		p_{J_1} = det(J_1-\lambda_1 I_n)
	\end{equation}
	where $\lambda_1$ represents eigenvalues associated with Eq.~\eqref{eig0_modeleq} (note that $\lambda_1=\lambda$; where $\lambda$ is the eigenvalue associated with Eq.~\eqref{ref:model_eq_1}, since both sets of equations represent the same system), and $I_n$ is the $n \times n$  Identity matrix.
	\begin{multline}
		\label{char_eq_subs}
		p_{J_1}= \lambda^{3} - \left(f_{m_1} + g_{m_2} + h_{b}\right) \lambda^{2}  \\ + \left(- f_{b} h_{m_1} + f_{m_1} g_{m_2} + f_{m_1} h_{b} - f_{m_2} g_{m_1} - g_{b} h_{m_2} + g_{m_2} h_{b}\right)\lambda \\-  f_{b} g_{m_1} h_{m_2} + f_{b} g_{m_2} h_{m_1} + f_{m_1} g_{b} h_{m_2}\\ - f_{m_1} g_{m_2} h_{b} - f_{m_2} g_{b} h_{m_1} + f_{m_2} g_{m_1} h_{b}  
	\end{multline}
	Also note that, when Eq.~\eqref{constraint} is differentiated with respect to time it gives;
	\begin{equation}
		\label{constraint_dif}
		\dot m_1+\dot m_2+\dot b=0
	\end{equation}
	By substituting Eq.~\eqref{constraint_dif} as $h=-f-g$ to Eq.~\eqref{char_eq_subs} and simplifying, it gives;
	\begin{equation}
		\label{approach1}
		p_{J_1}=  \lambda (\lambda^{2} - (f_{m_1}-f_{b}+g_{m_2}-g_b)\lambda+\left(f_{m_1}-f_{b}\right)\left(g_{m_2}-g_b\right)-\left(f_{m_2}-f_{b}\right)\left(g_{m_1}-g_{m_2})\right)
	\end{equation}
	
	\textit{Approach 2}: Presenting the model equations (see Eq.~\eqref{ref:model_eq_1}) as a two-dimensional system by substituting the condition in Eq.~\eqref{constraint} into the model equations:
	\begin{equation}
		\begin{aligned}
			\label{eig0_modeleq_linear}
			\dot m_1 &= f(m_1,m_2,1-m_1-m_2)\\
			\dot m_2 &= g(m_1,m_2,1-m_1-m_2)
		\end{aligned}
	\end{equation} 	
	The Jacobian matrix for the above set of equations is:
	\begin{equation}
		J_2=\left[\begin{matrix}
			f_{m_1}-f_b & f_{m_2}-f_b \\
			
			g_{m_1}-g_b & g_{m_2}-g_b
		\end{matrix}\right]
	\end{equation}
	where $b=1-m_1+m_2$	
	The eigenvalues of the above Jacobian matrix satisfies the following characteristic polynomial:
	
	\begin{equation}
		\label{char_eq2}
		p_{J_2} = det(J_2-\lambda_2 I_n)
	\end{equation}
	where $\lambda_2$ represents eigenvalues associated with the system depicted in Eq.~\eqref{eig0_modeleq_linear} (note that $\lambda_2=\lambda$; where $\lambda$ is the eigenvalue associated with Eq.~\eqref{ref:model_eq_1}, since both sets of equations represent the same system), and $I_n$ is the $n \times n$  Identity matrix.
	\begin{equation}
		\label{approach2}
		p_{J_2}=\lambda^2-\left(f_{m_1}-f_{b}+g_{m_2}-g_b\right) \lambda +\left(f_{m_1}-f_{b}\right)\left(g_{m_2}-g_b\right)-\left(f_{m_2}-f_{b}\right)\left(g_{m_1}-g_{m_2}\right)
	\end{equation}	
	When comparing Eq.~\eqref{approach1} and Eq.~\eqref{approach2}, it gives that;	
	\begin{equation}
		\label{rel_char}
		p_{J_1}=\lambda p_{J_2}
	\end{equation}
	The eigenvalues are obtained by setting the characteristic equation to 0. Then the equation Eq.~\eqref{rel_char} implies that the model equations presented as a three-dimensional system with an additional constraint $m_1+m_2+b=1$ (see Eq.~\eqref{eig0_modeleq}) has the same eigenvalues apart from $\lambda=0$, as the model equations presented as a two-dimensional system by incorporating the constraint $m_1+m_2+b=1$ into the model (see Eq.~\eqref{eig0_modeleq_linear}).
	
	Thus, it can be concluded that the eigenvalue equals to 0 (i.e. $\lambda=0$) corresponds to the constraint $m_1+m_2+b=1$, thus stability of the system (see Eq.~\eqref{ref:model_eq_1}) is determined by eigenvalues other than $0$.
	
\end{proof}

\subsection{Jacobian matrices}
\label{chap:appendixA2}


\begin{multline}
	J_7 =
	\left[\begin{array}{cc}
		\frac{\alpha s_{M_1} {m_1}^\alpha b^{\beta + 1}}{{m_1}} - \frac{s_{B} {m_1}^{\beta + 1} b^\alpha \left(b + 1\right)}{{m_1}} & 0 \\
		0 & \frac{\alpha s_{M_2} {m_2}^\alpha b^{\beta + 1}}{{m_2}} - \frac{s_{B} {m_2}^{\beta + 1} b^\alpha \left(b + 1\right)}{{m_2}}\\
		- \frac{\alpha s_{M_1} {m_1}^\alpha b^{\beta + 1}}{{m_1}} + \frac{s_{B} {m_1}^{\beta + 1} b^\alpha \left(b + 1\right)}{{m_1}} &
		- \frac{\alpha s_{M_2} {m_2}^\alpha b^{\beta + 1}}{{m_2}} + \frac{s_{B} {m_2}^{\beta + 1} b^\alpha \left(b + 1\right)}{{m_2}}
	\end{array}\right.
	\\
	\left.\begin{array}{c}
		- \frac{\alpha s_{B} {m_1}^{\beta + 1} b^\alpha}{b} + \frac{s_{M_1} {m_1}^\alpha b^{\beta + 1} \left(b + 1\right)}{b} \\
		- \frac{\alpha s_{B} {m_2}^{\beta + 1} b^\alpha}{b} + \frac{s_{M_2} {m_2}^\alpha b^{\beta + 1} \left(b + 1\right)}{b}
	\end{array}\right]
\end{multline}

\begin{multline}
	J_6 =
	\left[\begin{array}{cc}
		\frac{\alpha b^{\beta + 1} m_{2}^{\alpha} s_{M_2}}{m_{2}} - \frac{b^{\alpha} m_{2}^{\beta + 1} s_{B} \left(\beta + 1\right)}{m_{2}} \\
		- \frac{\alpha b^{\beta + 1} m_{2}^{\alpha} s_{M_2}}{m_{2}} + \frac{b^{\alpha} m_{2}^{\beta + 1} s_{B} \left(\beta + 1\right)}{m_{2}} 
	\end{array}\right.
	\\
	\left.\begin{array}{c}
		- \frac{\alpha b^{\alpha} m_{2}^{\beta + 1} s_{B}}{b} + \frac{b^{\beta + 1} m_{2}^{\alpha} s_{M_2} \left(\beta + 1\right)}{b}\\
		\frac{\alpha b^{\alpha} m_{1}^{\beta + 1} s_{B}}{b} + \frac{\alpha b^{\alpha} m_{2}^{\beta + 1} s_{B}}{b} - \frac{b^{\beta + 1} m_{1}^{\alpha} s_{M_1} \left(\beta + 1\right)}{b} - \frac{b^{\beta + 1} m_{2}^{\alpha} s_{M_2} \left(\beta + 1\right)}{b}
	\end{array}\right]
\end{multline}


\begin{multline}
	J_5 =
	\left[\begin{array}{cc}
		\frac{\alpha b^{\beta + 1} m_{1}^{\alpha} s_{M_1}}{m_{1}} - \frac{b^{\alpha} m_{1}^{\beta + 1} s_{B} \left(\beta + 1\right)}{m_{1}} \\
		- \frac{\alpha b^{\beta + 1} m_{1}^{\alpha} s_{M_1}}{m_{1}} + \frac{b^{\alpha} m_{1}^{\beta + 1} s_{B} \left(\beta + 1\right)}{m_{1}}  
	\end{array}\right.
	\\
	\left.\begin{array}{c}
		- \frac{\alpha b^{\alpha} m_{1}^{\beta + 1} s_{B}}{b} + \frac{b^{\beta + 1} m_{1}^{\alpha} s_{M_1} \left(\beta + 1\right)}{b}\\
		\frac{\alpha b^{\alpha} m_{1}^{\beta + 1} s_{B}}{b} + \frac{\alpha b^{\alpha} m_{2}^{\beta + 1} s_{B}}{b} - \frac{b^{\beta + 1} m_{1}^{\alpha} s_{M_1} \left(\beta + 1\right)}{b} - \frac{b^{\beta + 1} m_{2}^{\alpha} s_{M_2} \left(\beta + 1\right)}{b}
	\end{array}\right]
\end{multline}

\section{Appendix 2}
\label{stability;cond}
This Appendix section gives the conditions to which equilibria become stable.
\subsection{Stability conditions of $E_6$}

\begin{align}
	\left(\frac{1}{-\delta}\right)\left(\frac{s _{M_2}}{s_B}\right)^{-\delta}\left(\frac{s_{M_{2}}^{\frac{\beta+2}{-\alpha+\beta+1}}}{s_B^{\frac{\alpha+1}{-\alpha+\beta+1}}}+\frac{s_{M_{2}}^{\frac{\beta+1}{-\alpha+\beta+1}}}{s_{B}^{\frac{\alpha}{-\alpha+\beta+1}}}\right)<0 \notag \\
	\left(\frac{s_{M_2}}{s_B}\right)^{-\delta} > 0  \notag \\
	\left(\frac{s_{M_{2}}^{\frac{\beta+2}{-\alpha+\beta+1}}}{s_B^{\frac{\alpha+1}{-\alpha+\beta+1}}}+\frac{s_{M_{2}}^{\frac{\beta+1}{-\alpha+\beta+1}}}{s_{B}^{\frac{\alpha}{-\alpha+\beta+1}}}\right) > 0 	
\end{align}

\subsection{Stability conditions of $E_5$}

\begin{align}
	\left(\frac{1}{-\delta}\right)\left(\frac{s _{M_1}}{s_B}\right)^{-\delta}\left(\frac{s_{M_{1}}^{\frac{\beta+2}{-\alpha+\beta+1}}}{s_B^{\frac{\alpha+1}{-\alpha+\beta+1}}}+\frac{s_{M_{1}}^{\frac{\beta+1}{-\alpha+\beta+1}}}{s_{B}^{\frac{\alpha}{-\alpha+\beta+1}}}\right)<0 \notag \\
	\left(\frac{s_{M_1}}{s_B}\right)^{-\delta} > 0 \notag \\
	\left(\frac{s_{M_{1}}^{\frac{\beta+2}{-\alpha+\beta+1}}}{s_B^{\frac{\alpha+1}{-\alpha+\beta+1}}}+\frac{s_{M_{1}}^{\frac{\beta+1}{-\alpha+\beta+1}}}{s_{B}^{\frac{\alpha}{-\alpha+\beta+1}}}\right) > 0
\end{align}

\subsection{Stability conditions of $E_4$}

The Jacobian matrix of this equilibrium is;

\begin{equation}
	\label{JacMat6}
	J_4=\left[\begin{matrix}
		0 & 0 \\
		0 & 0
	\end{matrix}\right]
\end{equation}

The Jacobian matrix for this equilibrium is a zero matrix, with all its eigenvalues equal to zero. It indicates that the system does not change in any direction. The eigenvectors of the corresponding Jacobian matrix (see Eq.~\eqref{JacMat6}) are;

\begin{align}
	v_1=\left[\begin{array}{ll}
		0\\1
	\end{array}\right] \notag    \\
	v_2=\left[\begin{array}{ll}
		1\\0
	\end{array}\right]  
\end{align}

This can be considered an indication of an equilibrium, with sensitive dependence on initial conditions and causing the system to have multiple steady states. i.e., different initial states reach different stable equilibria but satisfy $m_1+m_2=1$.\\

\subsection{Stability conditions of $E_3$}
The equilibrium stably occurs when ${s_B}>{s_{M_1}}$ and $\alpha - \beta \rightarrow{1^{-}}$, making the following limits.

\begin{align}
	\lim_{\alpha-\beta \to 1^{-}} m_1^{*} = \lim_{\alpha-\beta \to 1^{-}}\left(\frac{(\frac{s_{M_1}}{s_B})^\delta}{1 + (\frac{s_{M_1}}{s_B})^\delta + (\frac{s_{M_2}}{s_B})^\delta}\right) = 0 \notag \\
	\lim_{\alpha-\beta \to 1^{-}} m_2^{*} = \lim_{\alpha-\beta \to 1^{-}}\left(\frac{(\frac{s_{M_2}}{s_B})^\delta}{1 + (\frac{s_{M_1}}{s_B})^\delta + (\frac{s_{M_2}}{s_B})^\delta}\right) = 0 \notag \\
	\lim_{\alpha-\beta \to 1^{-}} b^{*} = \lim_{\alpha-\beta \to 1^{-}}\left(\frac{1}{1 + (\frac{s_{M_1}}{s_B})^\delta + (\frac{s_{M_2}}{s_B})^\delta}\right) = 1
\end{align} 

\bibliographystyle{unsrt}
\bibliography{LD_preprint}

\end{document}